\documentclass[reqno, 12pt,a4]{amsart}

\usepackage[mathscr]{eucal}
\usepackage{amsmath}
\usepackage{graphicx}
\usepackage{amssymb}
\usepackage{latexsym}
\usepackage{amsthm}
\usepackage{amscd}
\usepackage{enumerate}
\theoremstyle{plain}
\newtheorem{thm}{Theorem}[section]
\newtheorem{prop}[thm]{Proposition}

\theoremstyle{definition}
\newtheorem{Def}[thm]{Definition}

\numberwithin{equation}{section}

\title{Inferior Regular Partitions and Glaisher Correspondence
}
\author{Masanori Ando\ \ (Nara Gakuen University)}
\date{}
\begin{document}
\pagestyle{empty}
\maketitle\thispagestyle{empty}
\section{Partition}
Let $n$ be a positive integer. A partition $\lambda $ is an integer sequence 
\[
\lambda =(\lambda_1,\lambda_2,\ldots,\lambda_\ell) 
\]
satisfying $\lambda _1 \geq \lambda _2 \geq \ldots \geq \lambda _\ell >0$. 
We call $\ell (\lambda):= \ell$ the length of $\lambda $, 
 $\displaystyle |\lambda |:=\sum _{i=1}^{\ell }{\lambda _i}$ the size of $\lambda $, 
and each $\lambda _i$ a part of $\lambda $. 
 We let ${\mathcal {P}} $ denote the set of partitions, $\mathcal{P}(n)$ the set of partitions size $n$. 
After this, ``$(n)$'' means the restriction of size $n$. 
%以降, 分割を成分とする集合に「$(n)$」を付けた際は「大きさ$n$」の制限がかかるものとする. 
For a partition $\lambda$, we let $m_i(\lambda )$ denote the multiplicity of $i$ as its part. 
$(1^{m_1(\lambda )}2^{m_2(\lambda )} \ldots  )$ is another representation of $\lambda $. 
\begin{Def}
For any positive integer $r\geq 2$, we define the next subsets of $\mathcal{P}$. 
\begin{align*}
&\mathcal{RP}_r:=\{ \lambda \in \mathcal{P}\ |\ {}^\forall k, m_k(\lambda )<r\}
\textrm{: the set of $r$-regular partitions}, \\
&\mathcal{CP}_r:=\{ \lambda \in \mathcal{P}\ |\ {}^\forall k, m_{rk}(\lambda )=0\}
\textrm{: the set of $r$-class regular partitions}, \\
&\mathcal{R'P}_r:=\{ \lambda \in \mathcal{P}\ |\ {}^{\exists !} k, m_k(\lambda )\geq r\}
\textrm{: the set of $r$-inferior regular partitions}. 
\end{align*}
\end{Def}
Except for $\mathcal{R'P}_r$, 
the generating functions of those are well known. 
\[
\sum_{\lambda \in \mathcal{P}}{q^{|\lambda |}}=\frac{1}{(q;q)_\infty }, 
\sum_{\lambda \in \mathcal{RP}_r}{q^{|\lambda |}}=\frac{(q^r;q^r)_\infty}{(q;q)_\infty}, 
\sum_{\lambda \in \mathcal{CP}_r}{q^{|\lambda |}}=\frac{(q^r;q^r)_\infty}{(q;q)_\infty}. 
\]
Here $(a;b)_k=(1-a)(1-ab)\cdots (1-ab^{k-1})$. 
Especially, $\sharp \mathcal{RP}_r(n)=\sharp \mathcal{CP}_r(n)$.
There is natural bijection between these two sets. 
That is {\it Glaisher correspondence} $g_r : \mathcal{CP}_r(n)\ni \lambda \longmapsto g_r(\lambda )\in \mathcal{RP}_r(n)$. 
If $\lambda$ has more than or equal to $r$ same size parts, 
we combine $r$ inside those. In other word, we replace $k^r$ by $rk$. 
Repeat these operations until the partition has come to an element of $\mathcal{RP}_r(n)$. 
And we denote $c_r(\lambda )$ the number of operations from $\lambda $ to $g_r(\lambda )$. 
We denote $c_{r,n}:=\sum_{\lambda \in \mathcal{CP}_r(n)}{c_r(\lambda )}$. 
\\
\textbf{Example. } $r=2, \lambda =(1^6)$
\[
(1^6)\mapsto (21^4)\mapsto (2^21^2)\mapsto (2^3)\mapsto (42)=g_2(1^6). 
\]
Then, $c_2(1^6)=4$. 
\begin{prop} For any positive integer $r\geq 2$, 
%任意の$2$以上の自然数$r$について, 
\[
\sum_{\lambda \in \mathcal{CP}_r}{c_r(\lambda )q^{|\lambda |}}
=\sum_{\lambda \in \mathcal{R'P}_r}{q^{|\lambda |}}
=\frac{(q^r;q^r)_\infty}{(q;q)_\infty}\sum_{k\geq 1}{\frac{q^{rk}}{1-q^{rk}}}. 
\]
%\footnote{
%$\sharp \mathcal{R'P}_r(n)=c_{r,n}$ の結果については２節での定理の証明からも得られる. 
%}
\end{prop}
\begin{proof}
For %the generation function of 
$r$-inferior regular partitions, 
%劣正則分割について, 次の様に分解して考える. 
\begin{eqnarray*}
&&\frac{(q^r;q^r)_\infty}{(q;q)_\infty}\cdot \frac{q^{rk}}{1-q^{rk}}\\
&=&\left( \prod_{i\not = k}\frac{1-q^{ri}}{1-q^i}\right)
\times \frac{1-q^{rk}}{1-q^k}\cdot \frac{q^{rk}}{1-q^{rk}}\\
&=&\left( \prod_{i\not = k}\frac{1-q^{ri}}{1-q^i}\right) \times (q^{rk}+q^{(r+1)k}+\cdots ). 
\end{eqnarray*}
Here, to choose a term from parentheses of right-hand side correspond to choose  more than or equal to $r$ of $k$ parts. 
%このとき, 右の括弧内で項を一つ選ぶことが, 成分$k$ を$r$個以上選ぶことに対応する. 
Then, 
%従って, 
\[
\frac{(q^r;q^r)_\infty}{(q;q)_\infty}\sum_{k\geq 1}{\frac{q^{rk}}{1-q^{rk}}}
\]
is the generating function of $r$-inferior regular partitions. 
%は1種の成分のみが$r$個以上となる分割の母関数となり, すなわち$r$-劣正則分割の母関数となる. 
On the other hand, for the generating function of $c_{r, n}$, 
%一方$c_{r,n}$ については, 
\begin{eqnarray*}
&&\frac{(q^r;q^r)_\infty}{(q;q)_\infty}\sum_{k\geq 1}\frac{q^{rk}}{1-q^{rk}}\\
&=&\frac{(q^r;q^r)_\infty}{(q;q)_\infty}\sum_{i\geq 1}\sum_{k, r \nmid k}\frac{q^{r^ik}}{1-q^{r^ik}}. \\
&=&\frac{(q^r;q^r)_\infty}{(q;q)_\infty}\sum_{i\geq 1}\sum_{k, r \nmid k}{(q^{r^ik}+q^{2r^ik}+\cdots +q^{mr^ik}+\cdots )}.
\end{eqnarray*}
%と分解して考えよう. 
Here, to choose $q^{mr^ik}$ term from parentheses correspond to 
choose a $r$-class regular partition that has more than or equal to $mr^i$ of $k$ parts. 
%このとき右側を展開した際の$q^{m\cdot r^i\cdot k}$ の項を選ぶことは, 
%成分$k$ が$m\cdot r^i$ 個以上含まれる類正則分割を一つ数えることになる. 
%これを$m$, $i$ について和をとるとASY におけるダイアグラムの大きさを数えることになり, 
When we take summation by $m$, the coefficient of $n$ is the number of Glaisher operations 
that make the $r^i$ multiple part. 
%$m$についての和をとればグレイシャーの$i$階の操作回数
And let take summation by $i$, it becomes the generating function of $c_{r,n}$. 
%, さらに$i$ について和をとるとグレイシャー全体の操作回数すなわち$c_{r,n}$の母関数となる
\end{proof}
For each Glaisher operation, the length of partition decreases just $r-1$. Then, 
%$g_r$ による操作$1$回につき, 分割の長さは$(r-1)$ 減少するため, 
\begin{prop} For any positive integer $n$, 
\[
\sum_{\lambda\in \mathcal{CP}_r(n)}{\ell (\lambda )}-\sum_{\lambda \in \mathcal{RP}_r(n)}{\ell (\lambda )}
=(r-1)c_{r,n}. 
\]
\end{prop}
We refine this identity. 
%このことの精密化を考えよう. 
\section{Mizukawa-Yamada's $X-Y=c$}
\begin{Def}
For $1\leq {}^\forall j\leq r-1, {}^\forall \lambda \in \mathcal{CP}_r, {}^\forall \mu \in \mathcal{RP}_r$, we define
\begin{eqnarray*}
x_{r,j}(\lambda )&:=\sharp \{ k\ |\ \lambda _k \equiv j (\textrm{mod}\ r)\}, &
X_{r,j,n}:= \sum_{\lambda \in \mathcal{CP}_r(n)}{x_{r,j}(\lambda )}, \\
y_{r,j}(\mu )&:=\sharp \{ k\ |\ m_k(\mu )\geq j\}, &
Y_{r,j,n}:= \sum_{\mu \in \mathcal{RP}_r(n)}{y_{r,j}(\mu )}. 
\end{eqnarray*}
\end{Def}
\begin{thm}[\cite{BOS}, \cite{MY}]
For any positive integer $r,j,n$, 
\[
X_{r,j,n}-Y_{r,j,n}=c_{r,n}. 
\]
Then the value of $X-Y$ doesn't depend on $j$. 
%特に$X-Y$ の値は$j$ に依らない. 
\end{thm}
\noindent
\textbf{Remark. }\ 
This is the refinement of \textbf{proposition 1.3.}
because
\[
\sum_{j=1}^{r-1}X_{r,j,n}=\sum_{\lambda \in \mathcal{CP}_r(n)}{\ell (\lambda)}, 
\sum_{j=1}^{r-1}Y_{r,j,n}=\sum_{\lambda \in \mathcal{RP}_r(n)}{\ell (\lambda)}. 
\]
%であるため, この命題が前章の最後の命題の精密化となる. \\
\\
\textbf{Example. }For $r=3, n=7$, 
\[
\mathcal{CP}_3(7)
=\{ (7), (52), (51^2), (421), (41^3), (2^31), (2^21^3), (21^5), (1^7)\}. 
\]
In the whole set, the number of part $7$ is $1$. 
Similarly, the number of part $5$ is $2$, part $4$ is $2$, part $2$ is $8$ and part $1$ is $22$. 
Then, 
%全体での各成分の個数を数えれば, $7$ が$1$個, $5$が$2$個, $4$が$2$個, $2$が$8$個, $1$が$22$個であり, 
\[
X_{3,1,7}=1+2+22=25, X_{3,2,7}=2+8=10. 
\]
On the other hand, 
\[
\mathcal{RP}_3(7)
=\{ (7), (61), (52), (51^2), (43), (421), (3^21), (32^2), (321^2)\}. 
\]
Then, 
\[
Y_{3,1,7}=19, Y_{3,2,7}=4. 
\]
The differences between $X$ and $Y$ are both $6$. 
\begin{proof}
We define
\[
A_{r,j,n}
:=\{ (\lambda ;k,\ell )\ |\ \lambda \in \mathcal{CP}_r(n), k\equiv j(\textrm{mod}\ r), 1\leq \ell \leq m_k(\lambda )\}. 
\]
From definition, $\sharp A_{r,j,n}=X_{r,j,n}$. 
\[
\begin{array}{cccc}
\varphi_{r,j} : &A_{r,j,n} &\longrightarrow &\mathcal{P}(n)\\
& \rotatebox{90}{$\in$} & & \rotatebox{90}{$\in$} \\
&(\lambda ;k,\ell )&\longmapsto &\mu 
\end{array}. 
\]
Here, $\mu =g_{r}(\lambda \setminus (k^\ell))\cup (\ell ^k)$. 
And $\cup, \setminus$ are sum and difference when we identify partition with multi-set. 
%は分割を多集合と見たときの和と差である. 
For $\mu$, the kind of part which have more than or equal to $r$ same part is at most one. 
%個数が$r$ 以上となる成分は高々1種であり, 
Then, 
\[
\varphi_{r,j}(A_{r,j,n})\subset \mathcal{RP}_r(n)\cup \mathcal{R'P}_r(n). 
\]
Let find multiplicity of each image. \\
%像毎の重複度を調べよう. \\
\underline{For $\mu \in \mathcal{RP}_r(n)$ }\\
We can make inverse image based on $\ell$ that $m_\ell(\mu)\geq j$. 
%個数が$j$ 以上の成分ごとに逆像が考えられるので, 
Then, 
\[
\sharp \varphi_{r,j}^{-1}(\mu)=y_{r,j}(\mu).
\] 
\underline{For $\mu \in \mathcal{R'P}_r(n)$ }\\
It is necessary to make it $r$-regular as preparations for inverse operation of Glaisher. 
%$r$-グレイシャーで戻す前に$r$-正則にする必要がある. 
For this, we decrease the part $\ell $ that $m_\ell(\mu)\geq r$. 
Because definition of $\mathcal{R'P}$, such $\ell$ decide unique. 
And it is unique too that $k$ which,  
\[
\mu \setminus (\ell )^k\in \mathcal{RP}_r(n), k \equiv j\ (\textrm{mod}\ r). 
\]
%そのために個数が$r$個以上となる唯一つの成分の個数を減らすわけだが, 
%$k$ 個減らした際に$r$-正則になるような, $r$を法として$j$と合同な$k$ はただ一通りである！
Then, $\sharp \varphi_{r,j}^{-1}(\mu)=1$. \\
\indent
Therefore $X_{r,j,n}=Y_{r,j,n}+\sharp \mathcal{R'P}_r(n)$, $X-Y$ does not depend on $j$. 
%よって, $X_{r,j,n}=Y_{r,j,n}+\sharp \mathcal{R'P}_r(n)$ となり, $X-Y$ は$j$ に依らず一定の値となる. 
\end{proof}
\section{Version for $\underline{r}$ . }
\begin{Def}
We think $\underline{r}=(r_1, r_2, \ldots , r_m)$ a $m$-tuple of positive integers greater than $1$. 
We assume that each elements in $\underline{r}$ are relatively prime. 
%ただし, 各成分は互いに素とする. このとき, 
We define
\[
\mathcal{CP}_{\underline{r}}:=
\{ \lambda \in \mathcal{P}\ |\ {}^\forall j,k, m_{r_jk}=0\}. 
\]
We call $\mathcal{CP}_{\underline{r}}$ the set of $\underline{r}$-class regular partitions. 
And we put $\underline{s}$ that tuple $(r_2, \ldots , r_m)$. 
%$\underline{r}$-類正則分割全体のなす集合と呼ぶ. また$\underline{s}=(r_2, r_3, \ldots , r_m)$ と置き, 
We define
\[
\mathcal{RP}_{\underline{r}}:=\mathcal{RP}_{r_1}\cap \mathcal{CP}_{\underline{s}}, 
\mathcal{R'P}_{\underline{r}}:=\mathcal{R'P}_{r_1}\cap \mathcal{CP}_{\underline{s}}. 
\]
We call these the set of $\underline{r}$-regular partitions and $\underline{r}$-inferior regular partitions. 
%$\underline{r}$-正則分割, $\underline{r}$-劣正則分割全体のなす集合と呼ぶ. 
\end{Def}
When following Mizukawa-Yamada\cite{MY}, 
the inclusion-exclusion principle gives us the generating functions of 
$\mathcal{CP}_{\underline{r}}$ and $\mathcal{RP}_{\underline{r}}$. 
\begin{prop}For any $m$-tuple $\underline{r}$, 
\[
\sum_{\lambda \in \mathcal{CP}_{\underline{r}}}q^{|\lambda |}
=\sum_{\lambda \in \mathcal{RP}_{\underline{r}}}q^{|\lambda |}
=\prod_{A\subset \underline{r}}(q^{\Pi A};q^{\Pi A})_\infty ^{(-1)^{|A|+1}}. 
\]
Here we identify $\underline{r}$ the set $\{ r_1, r_2, \ldots , r_m\}$. 
And $\Pi A:=\prod_{r\in A}r$. 
%ここで$\underline {r}$ を集合$\{ r_1, r_2, \ldots , r_m\}$ と同一視しており, 
\footnote
{
When $A=\emptyset$, $\Pi A=1$. 
}
%Then, for any positive integer $n$, 
%従って任意の自然数$n$ について, 
%\[
%\sharp \mathcal{CP}_{\underline{r}}(n)=\sharp \mathcal{RP}_{\underline{r}}(n).
%\]
\end{prop}
%\begin{proof}
%$\mathcal{CP}_{\underline{r}}$ について, 
%集合の位数の和公式
%\[
%\sharp (A\cup B\cup C)=\sharp A+\sharp B+\sharp C
%-(\sharp (A\cup B)+\sharp (B\cup C)+\sharp (C\cup A))+\sharp (A\cup B\cup C)
%\]
%と同じ原理. 
%すなわち, 
%\begin{eqnarray*}
%&&(\textrm{$r_1\ldots r_m$ の倍数を持たない分割})\\
%&=&(\textrm{分割全体})\\
%&-&((\textrm{$r_1$ の倍数を少なくとも一つもつ分割})+\cdots +(\textrm{$r_m$ の倍数を少なくとも一つもつ分割}))\\
%&-&((\textrm{$r_1$ の倍数と$r_2$ の倍数をどちらも少なくとも一つもつ分割})+\cdots )\\
%&+&\cdots \mp \cdots 
%\end{eqnarray*}
%$\mathcal{RP}_{\underline{r}}$ について, 
%\end{proof}
The inclusion-exclusion principle gives us only the generating function of 
$\mathcal{CP}_{\underline{r}}$ correctly. 
However, 
the $r_1$-Glaisher correspondence is 
also bijection between $\mathcal{CP}_{\underline{r}}(n)$ and $\mathcal{RP}_{\underline{r}}(n)$. 
Then the generating functions of both sets are equal. 
%集合間の全単射写像には$r_1$-グレイシャーがそのまま使える. 
We define the number of $r_1$-Glaisher operations over $\mathcal{CP}_{\underline{r}}(n)$,  
$c_{\underline{r},n}:=\sum_{\lambda \in \mathcal{CP}_{\underline{r}}(n)}{c_{r_1}}$. 
%$\underline{r}$-類正則全体での操作回数を, 
%$c_{\underline{r},n}:=\sum_{\lambda \in \mathcal{CP}_{\underline{r}}(n)}{c_{r_1}}$ で定義しておく. 
\begin{prop}For any $m$-tuple $\underline{r}$, 
%任意の組$\underline{r}$ について, 
\[
\sum_{\lambda \in \mathcal{CP}_{\underline{r}}(n)}{c_{r_1}(\lambda )q^{|\lambda |}}
=\sum_{\lambda \in \mathcal{R'P}_{\underline{r}}(n)}q^{|\lambda |}. 
\]
\end{prop}
\begin{proof}
We proof that both sides are equal to the next generating function. 
\[
\left( \prod_{A\subset \underline{r}}(q^{\Pi A};q^{\Pi A})_\infty^{(-1)^{|A|+1}}\right) 
\sum_{r_1\in A\subset \underline{r}}\sum_{k\geq 1}\frac{(-1)^{|A|+1}q^{k\Pi A}}{1-q^{k\Pi A}}. 
\]
From \textbf{proposition 3.2}, $\left( \prod_{A\subset \underline{r}}(q^{\Pi A};q^{\Pi A})_\infty^{(-1)^{|A|+1}}\right) $ 
is the generating function of $\mathcal{RP}_{\underline{r}}$. 
From the inclusion-exclusion principle, 
\[
\sum_{r_1\in A\subset \underline{r}}\sum_{k\geq 1}\frac{(-1)^{|A|+1}q^{k\Pi A}}{1-q^{k\Pi A}}
=\sum_{k\geq1, r_2, r_3, \ldots r_m \nmid k}{\frac{q^{r_1k}}{1-q^{r_1k}}}. 
\]
Similarly for the proof of proposition 1.2, to multiply the generation function of $\mathcal{RP}_{\underline{r}}$ by this $q$-series 
is correspond to increase the number of only one kind of part to $r_1$ and over. 
\end{proof}
%通常版同様$\sharp \mathcal{R'P}_{\underline{r}}(n)=c_{\underline{r},n}$は, 正則・類正則分割の長さの差に関しての命題の精密化から得られるが, 
%$\underline{r}$ 版での精密化はもう少し条件が必要なため, 「成分が互いに素」の条件のみだと今のところ母関数による証明となる. 
\begin{Def}
For $\underline{r}, 1\leq j\leq r_1, n$, we define
\[
X_{\underline{r},j,n}
:=\sum_{\lambda \in \mathcal{CP}_{\underline{r}}(n)}{x_{r_1,j}(\lambda )}, 
Y_{\underline {r},j,n}
:=\sum_{\mu \in \mathcal{RP}_{\underline{r}}(n)}{y_{r_1,j}(\mu )}. 
\]
%と定義する. 
\end{Def}
Here, $X-Y$ may not be independence from $j$. \\
%このとき, $X-Y$ は$j$に依らず一定になるとは限らない. \\
\textbf{Example. }$\underline{r}=(3,5), n=5$
\[
\mathcal{CP}_{(3,5)}(5)=\{ (41), (2^21), (21^3), (1^5)\}. 
\]
Then, $X_{(3,5), 1,5}=11, X_{(3,5), 2, 5}=3$. 
On the other hand, 
\[
\mathcal{RP}_{(3,5)}(5)=\{ (41), (32), (31^2), (2^21)\}. 
\]
Then, $Y_{(3,5), 1, 5}=8, Y_{(3,5), 2, 5}=2$. 
Therefore, when $j$ is different, $X-Y$ is also different. 
%となり, $j$が違えば$X-Y$ の値は異なる. 
\begin{thm}
For any tuple $\underline{r}=(r_1, r_2, \ldots , r_m), r_2, r_3, \ldots$ , $r_m\equiv 1({\rm{mod}}\ r_1)$, 
$1\leq j\leq r_1-1$ and positive integer $n$, 
%任意の組$\underline{r}=(r_1, r_2, \ldots , r_m), r_2, r_3, \ldots , r_m\equiv 1({\textrm{mod}} r_1)$ と, 
%$1\leq j\leq r_1-1$, 自然数$n$について, 
\[
X_{\underline{r}, j,n}-Y_{\underline{r}, j,n}=c_{\underline{r}, n}. 
\]
\end{thm}
\begin{proof}
Similarly the proof of last section, we define
%先ほどと同様, 
\[
A_{\underline{r}, j, n}:=
\{ (\lambda ;k,\ell )\ |\ \lambda \in \mathcal{CP}_{\underline{r}}(n), k\equiv j(\textrm{mod}\ r_1), 1\leq \ell \leq m_k(\lambda )\}. 
\]
%を考える. 
From definition, $\sharp A_{\underline{r},j,n}=X_{\underline{r},j,n}$. 
We construct a map 
\[
\begin{array}{cccc}
\varphi_{\underline{r},j} : &A_{\underline{r},j,n} &\longrightarrow &\mathcal{P}(n)\\
& \rotatebox{90}{$\in$} & & \rotatebox{90}{$\in$} \\
&(\lambda ;k,\ell )&\longmapsto &\mu 
\end{array}. 
\]
Here, $\mu = g_{r_1}(\lambda \setminus (k^\ell ))\cup ((\ell)_{\underline{s}'}^{k(\ell )_{\underline{s}}})$. 
$(\ell )_{\underline{s}}$ and $(\ell )_{\underline{s}'}$ are $\underline {s}$-part and $\underline{s}$-prime part of $\ell$. 
That is $\ell =(\ell )_{\underline{s}}\cdot (\ell )_{\underline{s}'}$ and 
$(\ell )_{\underline{s}}=r_2^{a_2}r_3^{a_3}\cdots r_m^{a_m}$, $r_2, r_3, \ldots , r_m \nmid (\ell )_{\underline{s}'}$. 
These are well-defined because $r_2, r_3, \ldots , r_m$ are relatively prime. 
%これらは, $r_2, r_3, \ldots , r_m$ が互いに素であることからwell-defined である. 
And from the definition of map, each image is $\underline{s}$-class regular 
and the kind of part which have more than or equal to $r_1$ same part is at most one. 
Then $\varphi_{\underline{r},j}(A_{\underline{r}, j, n})\subset 
\mathcal{RP}_{\underline{r}}(n)\cup \mathcal{R'P}_{\underline{r}}(n)$. 
%また像は_の部分集合となる. 
Let find multiplicity of image. 
%像の重複度を調べよう. 
%その前に, 
When fix $\ell$ that $m_\ell(\mu) \geq j$. 
There is unique $k$ that 
the partition $g_{r_1}^{-1}(\mu\setminus (\ell)^k)\cup (k)_{\underline{s'}}^{\ell (k)_{\underline{s}}}$ 
is the element of $\varphi_{\underline{r}, j}^{-1}(\mu)$. 
Here because of $r_2, r_3, \ldots r_m \equiv 1 (\textrm{mod}\ r_1)$, $k\equiv (k)_{\underline{s'}}\ (\textrm{mod}\ r_1)$. \\
\underline{For $\mu \in \mathcal{RP}_{\underline{r}}(n)$ }\\
We can make inverse image based on $\ell$ such that $m_\ell(\mu)\geq j$. 
%個数が$j$ 以上の成分ごとに逆像が考えられるので, 
Then, $\sharp \varphi_{\underline{r},j}^{-1}(\mu)=y_{\underline{r},j}(\mu)$. \\
\underline{For $\mu \in \mathcal{R'P}_{\underline{r}}(n)$ }\\
It is necessary to make it $\underline{r}$-regular as preparations for inverse of Glaisher. 
%$r$-グレイシャーで戻す前に$r$-正則にする必要がある. 
For this, we decrease the part $\ell $ that $m_\ell(\mu)\geq r$. 
Because definition of $\mathcal{R'P}$, such $\ell$ decide unique. 
And it is unique too that $k$ which,  
\[
\mu \setminus (\ell )^k\in \mathcal{RP}_{\underline{r}}(n), k \equiv j\ (\textrm{mod}\ r_1). 
\]
%そのために個数が$r$個以上となる唯一つの成分の個数を減らすわけだが, 
%$k$ 個減らした際に$r$-正則になるような, $r$を法として$j$と合同な$k$ はただ一通りである！
Then, $\sharp \varphi_{\underline{r},j}^{-1}(\mu)=1$. \\
\end{proof}

\end{document}